\documentclass{amsart}

\usepackage{amssymb}

\usepackage[dvips]{graphicx}
\usepackage{psfrag}

\theoremstyle{plain}
\newtheorem{mainthm}{Theorem}

\swapnumbers
\newtheorem{theorem}{Theorem}[section]

\newtheorem{corollary}[theorem]{Corollary}
\newtheorem{lemma}[theorem]{Lemma}
\newtheorem{conj}[theorem]{Conjecture}

\newcommand{\tcap}{\pitchfork}

\newcommand{\mundo}{\operatorname{Diff}^1(M)}

\newcommand{\Z}{\mathbb{Z}}
\newcommand{\N}{\mathbb{N}}

\newcommand{\eps}{\varepsilon}

\title[Two-sided limit shadowing property]{Hyperbolicity, transitivity and the two-sided limit shadowing property}

\author[Bernardo Carvalho]{Bernardo Carvalho}

\date{\today}

\thanks{This paper was partially supported by CAPES (Brazil).}

\address{Universidade Federal do Rio de Janeiro - UFRJ}

\email{bmcarvalho06@gmail.com}

\begin{document}

\begin{abstract}
We explore the notion of two-sided limit shadowing property introduced by Pilyugin \cite{P1}. Indeed, we characterize the $C^1$-interior of the set of diffeomorphisms with such a property on closed manifolds as the set of transitive Anosov diffeomorphisms. As a consequence we obtain that all codimention-one Anosov diffeomorphisms have the two-sided limit shadowing property. We also prove that every diffeomorphism $f$ with such a property on a closed manifold has neither sinks nor sources and is transitive Anosov (in the Axiom A case). In particular, no Morse-Smale diffeomorphism have the two-sided limit shadowing property. Finally, we prove that $C^1$-generic diffeomorphisms on closed manifolds with the two-sided limit shadowing property are transitive Anosov. All these results allow us to reduce the well-known conjecture about the transitivity of Anosov diffeomorphisms on closed manifolds to prove that
the set of diffeomorphisms with the two-sided limit shadowing property coincides with the set of Anosov diffeomorphisms.
\end{abstract}

\maketitle

\section{Introduction and Statement of Results}

Many kinds of \textit{shadowing} properties were intensely studied last years. Specially the pseudo-orbit shadowing property (see \cite{P}) due to its relationship with stability and ergodic theories. This property studies the closeness of approximate and exact orbits of dynamical systems on unbounded intervals. We consider $(X,d)$ a compact metric space and $f:X\rightarrow X$ a homeomorphism. The orbit of a point $x\in X$ is the set $\{f^i(x); i\in\Z\}$. Fix $\delta>0$. We say that a sequence $\{x_i\}_{i\in\Z}$ of points in $X$ is a \emph{$\delta$-pseudo-orbit} if it satisfies
$$d(f(x_i),x_{i+1})<\delta, \,\,\,\,\,\, i\in\Z.$$

A pseudo-orbit is $\eps$-\emph{shadowed} if there is a point $y\in X$ such that
$$d(f^i(y),x_i)<\eps, \,\,\,\,\,\, i\in\Z.$$

We say that $f$ has the \emph{shadowing property} (usually called pseudo-orbit tracing property) if it satisfies: given $\eps>0$ there exists $\delta>0$ such that every $\delta$-pseudo-orbit is $\eps$-shadowed. This property is well studied in \cite{Pa} and \cite{P}. Often pseudo-orbits are obtained as results of numerical studies of dynamical systems. In this context the shadowing property means that numerically found orbits with uniform small errors are close to real orbits.

Another kind of shadowing property that has been intensely studied is the \emph{limit shadowing property}. It deals with pseudo-orbits with errors tending to zero. More precisely, we say that a sequence $(x_i)_{i\in\N}$ of points in $X$ is a \emph{limit pseudo-orbit} if it satisfies
$$d(f(x_i),x_{i+1})\rightarrow 0, \,\,\,\,\,\, i\rightarrow\infty.$$

A limit pseudo-orbit is \emph{limit-shadowed} if there exists a point $y\in X$ such that
$$d(f^i(y),x_i)\rightarrow 0, \,\,\,\,\,\, i\rightarrow \infty.$$

We say that $f$ has the \emph{limit shadowing property} if every limit pseudo-orbit is limit-shadowed. In this case the values $d(f(x_i),x_{i+1})$ may be large for small $i\in\N$ but converge to zero as $i\rightarrow\infty$. From the numerical viewpoint this property means the following (as observed in \cite{ENP}): if we apply a numerical method that approximate $f$ with `improving accuracy' so that one step errors tend to zero as time goes to infinity then the numerically obtained orbits tend to real ones.

We define a negative limit shadowing property as follows. A sequence $(x_i)_{i\leq0}$ of points in $X$ is a \emph{negative limit pseudo-orbit} if it satisfies
$$d(f(x_i),x_{i+1})\rightarrow 0, \,\,\,\,\,\, i\rightarrow-\infty.$$
A negative limit pseudo-orbit is \emph{limit-shadowed} if there exists a point $y\in X$ such that
$$d(f^i(y),x_i)\rightarrow 0, \,\,\,\,\,\, i\rightarrow-\infty.$$
We say that $f$ has the \emph{negative limit shadowing property} if every negative limit pseudo-orbit is limit-shadowed.





In this work we consider an analogue property that consider two-sided limit pseudo-orbits and two-sided limit shadows. Precisely, we say that a sequence $(x_i)_{i\in\Z}$ of points in $X$ is a \emph{two-sided limit pseudo-orbit} if it satisfies
$$d(f(x_i),x_{i+1})\rightarrow 0, \,\,\,\,\,\, |i|\rightarrow\infty.$$

A two-sided limit pseudo-orbit is \emph{two-sided limit-shadowed} if there is a point $y\in X$ such that
$$d(f^i(y),x_i)\rightarrow 0, \,\,\,\,\,\, |i|\rightarrow \infty.$$

We say that $f$ has the \emph{two-sided limit shadowing property} if every two-sided limit pseudo-orbit is two-sided limit shadowed.

Let $M$ be a closed $C^{\infty}$ manifold and let $\mundo$ be the set of all diffeomorphisms of $M$ endowed with the $C^1$ topology. Denote by $d$ the distance on $M$ induced from a Riemannian metric $\|.\|$ on the tangent bundle $TM$.

We denote by $\mathcal{S}$ the set of all diffeomorphisms of $M$ with the shadowing property, denote by $\mathcal{LS}$ the set of all diffeomorphisms of $M$ with the limit shadowing property and denote by $\mathcal{TLS}$ the set of all diffeomorphisms of $M$ with the two-sided limit shadowing property.

S. Pilyugin says that it is unreasonable to study the two-sided limit shadowing property without putting any restriction on the values $d(f(x_i),x_{i+1})$ (see Remark 2 \cite{P1}). Counteracting him we not only study the two-sided limit shadowing property without any restriction on the values $d(f(x_i),x_{i+1})$, but we are able to characterize the $C^1$ interior of $\mathcal{TLS}$ and relate this to a well known open problem.

Let $\Lambda$ be a compact and invariant set. A linear subbundle $E$ of the tangent bundle $T_{\Lambda}M$ is \emph{uniformly contracted} by $f$ if it is $Df$-invariant and there exists $N\geq1$ such that for any $x\in K$ and any unit vector $v\in E_x$ we have $$||Df^N(x)v||<\frac{1}{2}.$$
If $E$ is uniformly contracted for $f^{-1}$ we say that it is \emph{uniformly expanded}. The set $\Lambda$ is \emph{hyperbolic} if $T_{\Lambda}M=E\oplus F$, where $E$ is uniformly contracted and $F$ is uniformly expanded. We call $E$ the \textit{stable bundle} and $F$ the \textit{unstable bundle}. If the whole ambient manifold $M$ is hyperbolic we say that $f$ is \emph{Anosov}. We say that $f$ is \emph{transitive} if it has a dense orbit in $M$.

\begin{mainthm}\label{maintheorem2}
The $C^1$-interior of $\mathcal{TLS}$ is equal to the set of transitive Anosov diffeomorphisms.
\end{mainthm}

This theorem can be related with other results concerning the $C^1$ interior of other dynamical properties: \cite{LMS} shows that the $C^1$ interior of $\mathcal{S}$ is equal to the set of $C^1$ structurally stable diffeomorphisms and \cite{P1} shows that the $C^1$ interior of $\mathcal{LS}$ is equal to the set of $\Omega$-stable diffeomorphisms.

This theorem can also be related with a very old open problem (see \cite{HK} section 18) that is whether Anosov diffeomorphisms are transitive. Many authors tried to prove this but only got partial results. J. Franks \cite{F1} and S. Newhouse \cite{N} proved it for codimension one Anosov diffeomorphisms (with dimension of stable or unstable space equal to one). If we could prove that Anosov diffeomorphisms belong to $\mathcal{TLS}$ then they belong to the $C^1$ interior of $\mathcal{TLS}$ (the set of Anosov diffeomorphisms is open in $\mundo$) and Theorem \ref{maintheorem2} assures that they are transitive. On the other hand, if an Anosov diffeomorphism is transitive then Theorem \ref{maintheorem2} shows that it belongs to $\mathcal{TLS}$. This allow us to state the following:

\begin{conj}\label{conj}
$\mathcal{TLS}$ is equal to the set of Anosov diffeomorphisms.
\end{conj}

As already mentioned, this reformulates the problem of the transitivity of Anosov diffeomorphisms in terms of the two-sided limit shadowing property. Using the results of Franks and Newhouse and also Theorem \ref{maintheorem2} we obtain a particular case of this conjecture:

\begin{corollary}
Codimension one Anosov diffeomorphisms have the two-sided limit shadowing property.
\end{corollary}

Another problem that is included in Conjecture \ref{conj} is when diffeomorphisms in $\mathcal{TLS}$ are Anosov. Theorem $\ref{maintheorem2}$ shows that diffeomorphisms in the $C^1$ interior of $\mathcal{TLS}$ are Anosov but unfortunately we do not know if $\mathcal{TLS}$ is open in $\mundo$.

In contrast, although Anosov diffeomorphisms have the shadowing property (see \cite{R}) it is well known that $\mathcal{S}$ is not contained in the set of Anosov diffeomorphisms. Indeed, Morse-Smale diffeomorphisms are $C^1$ structurally stable (see \cite{PS}) then belong to the $C^1$ interior of $\mathcal{S}$ but are not Anosov. Moreover, $\mathcal{S}$ is not open in $\mundo$. J. Lewowicz \cite{L} gives an example of a diffeomorphism that is topologically stable (thus has the shadowing property if dim$(M)\geq2$ see \cite{W}) but is not $C^1$ structurally stable then do not belong to the $C^1$ interior of $\mathcal{S}$.

Another partial answer to Conjecture \ref{conj} is as follows. Let $f\in\mundo$. A point $p\in M$ is called \textit{periodic} if there exists $n\in\N$ such that $f^n(p)=p$. The smallest $n\in\N$ satisfying $f^n(p)=p$ is called the period of $p$ and will be denoted by $\pi(p)$. Denote by $Per(f)$ the set of all periodic points of $f$. A periodic point is called \textit{hyperbolic} if its orbit $\mathcal{O}(p)$ is a hyperbolic set. We define the \textit{index} of a hyperbolic periodic point $p$ as the dimension of the stable bundle and denote it by $ind(p)$. If $ind(p)=dim(M)$ we say that $p$ is a \emph{sink} and if $ind(p)=0$ we say that $p$ is a \emph{source}.

A point $p\in M$ is called \textit{non-wandering} if for every neighborhood $U$ of $p$ there exists $n\in\N$ satisfying $f^n(U)\cap U\neq\emptyset$. Denote by $\Omega(f)$ the set all non-wandering points of $f$. We say that a diffeomorphism $f$ is \emph{Morse-Smale} if its nonwandering set is
a finite union of hyperbolic periodic points all of whose invariant manifolds are in general position.

If $\Omega(f)=\overline{Per(f)}$ and is a hyperbolic set we say that $f$ is \textit{Axiom A}. Toward to characterize the set $\mathcal{TLS}$, we prove the following:

\begin{mainthm}\label{depois}
If $f\in\mathcal{TLS}$ then $f$ has neither sinks nor sources. If, in addition, $f$ is Axiom A, then $f$ is transitive Anosov. In particular, Morse-Smale diffeomorphisms do not belong to $\mathcal{TLS}$.
\end{mainthm}

More generally this shows that Axiom A diffeomorphisms that are not transitive Anosov do not have the two-sided limit shadowing property. It remains to know the relation of the two-sided limit shadowing and non-Axiom A diffeomorphisms.

We also study $C^1$ generic diffeomorphisms with the two-sided limit shadowing property. F. Abdenur and L. Diaz in \cite{AD} conjectured that $C^1$ generic diffeomorphisms with the shadowing property are $C^1$ structurally stable. We prove an anologue of this conjecture for the two-sided limit shadowing property. We say that a subset $\mathcal{R}$ of $\mundo$ is \emph{residual} if it contains a countable intersection of open and dense subsets of $\mundo$.

\begin{mainthm}\label{maintheorem1}
There exists a residual subset $\mathcal{R}$ of $\mundo$ such that every $f\in\mathcal{R}\cap\mathcal{TLS}$ is a transitive Anosov (and so structurally stable) diffeomorphism.
\end{mainthm}

We observe that D. Todorov proved some results in the stability theory \cite{T} using a property called \emph{Lipschtiz two-sided limit shadowing property with exponent $\gamma$}. He proved that for all $\gamma\geq0$ this property is equivalent to structural stability (Theorem 4 in \cite{T}). As did Pilyugin he put restrictions on the values $d(f(x_i),x_{i+1})$. The biggest difference of this paper is that we consider the two-sided limit shadowing property without any restriction on these values.

This paper is organized as follows. In section 2 we discuss some topological properties, their consequences and relate them with the two-sided limit shadowing property. We also prove some consequences of the two-sided limit shadowing property and prove Theorem \ref{depois}. Then in section 3 we prove Theorems \ref{maintheorem2} and \ref{maintheorem1}.

\section{Some topological properties and proof of Theorem \ref{depois}}

Consider $(X,d)$ a compact metric space. We say that a homeomorphism $f:X\rightarrow X$ is \emph{expansive} if there exists $\eps>0$ such that if $x,y\in M$ satisfy $d(f^i(x),f^i(y))<\eps$ for all $i\in\Z$ then $x=y$. This says that two different orbits move away from each other.

For any $y\in X$ and $\eps>0$ we define
$$W^s_{\eps}(y)=\{x\in M; d(f^n(x),f^n(y))<\eps \,\,\, \textrm{for every} \,\,\, n\in\N\}$$
$$W^u_{\eps}(y)=\{x\in M; d(f^{-n}(x),f^{-n}(y))<\eps \,\,\, \textrm{for every} \,\,\, n\in\N\}.$$
These sets are called $\eps$-stable and $\eps$-unstable sets of $y$ respectively. More generally define the stable and unstable sets of $y\in X$ as
$$W^s(y)=\{x\in M; d(f^n(x),f^n(y))\rightarrow0, \,\,\, \textrm{if} \,\,\, n\rightarrow\infty\}$$
$$W^u(y)=\{x\in M; d(f^{-n}(x),f^{-n}(y))\rightarrow0, \,\,\, \textrm{if} \,\,\, n\rightarrow\infty\}.$$
It is well known that if $f$ is expansive then there exists $\eps>0$ such that $W^s_{\eps}(y)\subset W^s(y)$ and $W^u_{\eps}(y)\subset W^u(y)$ for all $y\in X$ (see \cite{M} Lemma 1).

Following the proof of Theorem 2.1 in \cite{ENP} we prove the following:

\begin{lemma}\label{shadexp}
Every expansive homeomorphism $f:X\rightarrow X$ with the shadowing property has both the limit shadowing and negative limit shadowing properties.
\end{lemma}

\begin{proof}
Since $f$ is expansive there exists $\eps>0$ such that $W^s_{\eps}(y)\subset W^s(y)$ and $W^u_{\eps}(y)\subset W^u(y)$ for all $y\in X$. Choose $0<\delta<\frac{\eps}{2}$ such that every $\delta$-pseudo-orbits is $\eps$-shadowed. Let $(x_n)_{n\in\N}$ be a limit pseudo-orbit. For each $j\in\N$ choose $n_j\in\N$ such that $$d(f(x_n),x_{n+1})<\frac{\delta}{j}, \,\,\,\, n>n_j.$$ The shadowing property assures the existence of points $y_j\in X$ such that $$d(f^n(y_j),x_n)<\frac{\delta}{j}, \,\,\,\, n>n_j.$$ We claim that $(x_n)_{n\in\N}$ is limit shadowed by $y_1$. Indeed, $$d(f^n(y_1),f^n(y_j))\leq d(f^n(y_1),x_n) + d(x_n,f^n(y_j))<2\delta<\eps, \,\,\,\, n>n_j.$$ By the choice of $\eps$ we obtain $$d(f^n(y_1),f^n(y_j))\rightarrow0, \,\,\,\, n\rightarrow\infty$$ for all $j\in\N$. This imply that $$d(f^n(y_1),x_n)\rightarrow0, \,\,\,\, n\rightarrow\infty.$$ Thus every limit pseudo-orbit is limit shadowed and $f$ has the limit shadowing property. We can similarly prove that $f$ has the negative limit shadowing property.
\end{proof}

Now we define the \emph{specification property}. First, we recall what a specification is. A pair $(\tau,P)$ is a \emph{specification} if it consists of a finite collection $\tau=\{I_1,\dots,I_m\}$ of finite intervals $I_i=[a_i,b_i]\subset\Z$ and a map $P:\bigcup_{i=1}^mI_i\rightarrow X$ such that for each $t_1, t_2\in I\in\tau$ we have $$f^{t_2-t_1}(P(t_1))=P(t_2).$$ The specification $(\tau,P)$ is said to be $L$-spaced if $a_{i+1}\geq b_i+L$ for all $i\in\{1,\dots,m\}$. Moreover, it is $\eps$-shadowed by $y\in X$ if $$d(f^n(y),P(n))<\eps \,\,\,\,\,\, \textrm{for every} \,\,\, n\in \bigcup_{i=1}^mI_i.$$
We say that a homeomorphism $f:X\rightarrow X$ has the \textit{specification property} if for every $\eps>0$ there exists $L\in\N$ such that every $L$-spaced specification is $\eps$-shadowed. It is well known (see \cite{DGS}) that diffeomorphisms with the specification property are \emph{topologically mixing}, i.e., for every two open sets $U$ and $V$ there exists $N\in\N$ such that $f^n(U)\cap V\neq\emptyset$ for all $n\geq N$. In particular, they are topologically transitive.

We prove the following.

\begin{lemma}\label{lematopologico}
Every expansive homeomorphism $f:X\rightarrow X$ with the shadowing and specification properties has the two-sided limit shadowing property.
\end{lemma}

\begin{proof}
Let $\{x_n\}_{n\in\Z}$ be a two-sided limit pseudo-orbit. We proved in Lemma \ref{shadexp} that $f$ has limit shadowing and negative limit shadowing. Thus there exists points $p_1,p_2\in X$ such that $d(f^n(p_1),x_n)\rightarrow0$ when $n\rightarrow-\infty$ and $d(f^n(p_2),x_n)\rightarrow0$ when $n\rightarrow\infty$.

The expansiveness ensures the existence of $\eps>0$ such that $W^s_{\eps}(y)\subset W^s(y)$ and $W^u_{\eps}(y)\subset W^u(y)$ for all $y\in X$. As $f$ has the shadowing property and the specification property we obtain $\delta>0$ and $L\in\N$ such that every $\delta$-pseudo-orbits is $\eps$-shadowed and every $L$-spaced specification is $\delta$-shadowed.

Choose $N\in\N$ such that $2N\geq L$, $d(f^{-n}(p_1),x_{-n})<\delta$ and $d(f^n(p_2),x_n)<\delta$ for all $n\geq N$. Define $I_1=\{-N\}$, $I_2=\{N\}$, $P(-N)=f^{-N}(p_1)$ and $P(N)=f^N(p_2)$. Then $(\{I_1,I_2\},P)$ is a $L$-spaced specification and there exists a point $z\in X$ satisfying $$d(f^{-N}(z),f^{-N}(p_1))=d(f^{-N}(z),P(-N))<\delta$$ and $$d(f^N(z),f^N(p_2))=d(f^N(z),P(N))<\delta.$$ This implies that the sequence $(y_n)_{n\in\Z}$ defined by
$$y_n=f^n(p_1), \,\,\, n\leq-N$$
$$y_n=f^n(z), \,\,\, -N<n<N$$
$$y_n=f^n(p_2), \,\,\, n\geq N$$
is a $\delta$-pseudo orbit. Then there exists a point $y\in X$ such that $d(f^n(y),f^n(p_1))<\eps$ for all $n\leq-N$ and $d(f^n(y),f^n(p_2))<\eps$ for all $n\geq N$. By the choice of $\eps$ we obtain that $d(f^n(y),f^n(p_1))\rightarrow0$ when $n\rightarrow-\infty$ and $d(f^n(y),f^n(p_2))\rightarrow0$ when $n\rightarrow\infty$. By the choice of $p_1$ and $p_2$ we obtain that $d(f^n(y),x_n)\rightarrow0$ when $|n|\rightarrow\infty$.
\end{proof}

It is well known that Anosov diffeomorphisms have the shadowing and expansiveness properties (see \cite{R}). As noted in the introduction it is not known if Anosov diffeomorphisms are transitive and much less if they have the specification property. If we consider a connected and compact manifold $M$, transitive Anosov diffeomorphisms are topologically mixing and thus have the specification property (see \cite{DGS}). So transitive Anosov diffeomorphisms have all the topological properties of the hypothesis of Lemma \ref{lematopologico} and we obtain the following:

\begin{corollary}\label{tsls}
Transitive Anosov diffeomorphisms have the two-sided limit shadowing property.
\end{corollary}

Now we obtain some consequences of the two-sided limit shadowing property.

\begin{lemma}\label{limsha}
If $f:X\rightarrow X$ is a homeomorphism with the two-sided limit shadowing property, then $W^s(x)\cap W^u(y)\neq\emptyset$ for all $x,y\in X$.
\end{lemma}

\begin{proof}
Consider the following sequence:
$$x_n=f^n(y), \,\,\, n\leq0$$
$$x_n=f^n(x), \,\,\, n>0.$$

As $d(f(x_n),x_{n+1})=0$ for $n<0$ and $n>0$, the sequence $(x_n)_{n\in\Z}$ is a two-sided limit pseudo-orbit. The two-sided limit shadowing property assures the existence of a point $z\in X$ such that $d(f^n(z),x_n)\rightarrow0$ when $|n|\rightarrow\infty$, i.e.,
\begin{eqnarray*}
d(f^n(z),f^n(y))\rightarrow 0, \,\,\,\,\,\,\,\,\, n\rightarrow-\infty
\end{eqnarray*}
\begin{eqnarray*}
d(f^n(z),f^n(x))\rightarrow0, \,\,\,\,\,\,\,\,\, n\rightarrow\infty
\end{eqnarray*}
This implies that $z\in W^s(x)\cap W^u(y)$.
\end{proof}

Let $\Lambda$ be a compact invariant set. We say that $\Lambda$ is an \emph{attracting set} of $f$ if there exists a neighborhood $U$ of $\Lambda$ such that $f(\overline{U})\subset U$ and $\Lambda=\bigcap_{n\in\N}f^n(\overline{U})$. We say that $\Lambda$ is a \emph{repelling set} of $f$ if it is an attracting set for $f^{-1}$. As a corollary of Lemma \ref{limsha} we obtain:

\begin{corollary}\label{Morales}
If $f:X\rightarrow X$ is a homeomorphism with the two-sided limit shadowing property, $A$ is an attracting set and $B$ is a repelling set of $f$ then $A=B=X$.
\end{corollary}

\begin{proof}
Let $U$ be a neighborhood of $A$ such that $f(\overline{U})\subset U$ and $A=\bigcap_{n\in\N}f^n(\overline{U})$. If $U=X$ then clearly $A=U=X$. Suppose then that $U\varsubsetneq X$. Let $A^{\ast}=\bigcap_{n\in\N}f^{-n}(M\setminus \overline{U})$. Then $A^{\ast}$ is a repelling set for $f$. Let $a\in A$ and $a^{\ast}\in A^{\ast}$. Lemma \ref{limsha} assures the existence of $y\in W^s(a^{\ast})\cap W^u(a)$. Then there exists $N\in\N$ such that $f^{-N}(y)\in U$ and $f^N(y)\in M\setminus \overline{U}$. Thus $f^{-N}(y)$ is a point in $U$ such that its future orbit does not belong to $U$. This contradicts the definition of $U$ and imply that $A=U=X$. We can similarly prove that $B=X$.
\end{proof}

\vspace{+0.4cm}

\hspace{-0.4cm}\textit{Proof of Theorem \ref{depois}}: Let $f\in\mathcal{TLS}$. We will show that the stable manifold of any sink $s$ of $f$ is the whole manifold. This is obviously an absurd because there are no contracting diffeomorphisms on a closed manifold. Suppose, by contradiction, that there exists a point $x\in M$ that do not belong to $W^s(s)$. Lemma \ref{limsha} implies that there exists $y\in W^u(s)\cap W^s(x)$. As $x\notin W^s(s)$ we have $y\neq s$. Then we obtained a point $y\neq s$ that belongs to $W^u(s)$. This is an absurd because the unstable manifold of any sink $s\in M$ is the unitary set $\{s\}$. Thus $W^s(s)=M$. This implies that $f$ do not have any sink. We can similarly prove that $f$ do not have any source.

Now, suppose that $f$ is Axiom A. The Spectral Decomposition Theorem (see \cite{HK} Theorem 18.3.1) says that $\Omega(f)$ is decomposed into a finite number of disjoint, transitive, hyperbolic and isolated sets called \emph{basic sets}. It is known that one of the basic sets must be an attracting set an one must be a repelling set. Corollary \ref{Morales} implies that the whole manifold is a basic set. In particular, $f$ is a transitive Anosov diffeomorphism.

Since the Morse-Smale diffeomorphisms are Axiom A but not Anosov, we obtain that they cannot belong to $\mathcal{TLS}$. \qed

\section{Proof of Theorems \ref{maintheorem2} and \ref{maintheorem1}}

In this section we are interested in obtain hyperbolicity from the two-sided limit-shadowing property. We will do this from both robust and generic viewpoints. We define the stable and unstable sets for a basic set $\Lambda$ as
$$W^s(\Lambda)=\{y\in M; d(f^n(y),\Lambda)\rightarrow0, \,\,\, n\rightarrow\infty\}$$
$$W^u(\Lambda)=\{y\in M; d(f^n(y),\Lambda)\rightarrow0, \,\,\, n\rightarrow-\infty\}.$$
We define a relation on the basic sets as $\Lambda_i>\Lambda_j$ if $$(W^s(\Lambda_i)\setminus\Lambda_i)\cap W^u(\Lambda_j)\neq\emptyset.$$ We say that $f$ satisfies the \textit{no-cycle condition} if $\Lambda_{i_0}>\Lambda_{i_1}>\dots>\Lambda_{i_j}>\Lambda_{i_0}$ is impossible among the basic sets. If $f$ is Axiom A and satisfies the no-cycle condition we say that $f$ is \textit{hyperbolic}.

In \cite{H} S. Hayashi proved a sufficient condition for a diffeomorphism to be hyperbolic. It is a well known condition and is called \emph{star}. We say that $f\in\mundo$ is \textit{star} if there exists a $C^1$ neighborhood $\mathcal{U}$ of $f$ such that every periodic point of every $g\in\mathcal{U}$ is hyperbolic. Actually, it is known that hyperbolicity, $\Omega$-stability and the star condition are equivalent. Our proofs are based on this fact and we will be interested in proving that the two-sided limit shadowing property implies the star condition.

We state a well known lemma in dynamics that bifurcates a non-hyperbolic periodic point into two distinct hyperbolic periodic points with different indices. It is proved in \cite{SSY} Lemma 2.4.

\begin{lemma}\label{pert}
Let $f\in\mundo$ and $p$ be a non-hyperbolic periodic point of $f$. Then for every $C^1$ neighborhood $\mathcal{U}$ of $f$ there are $g\in\mathcal{U}$ and $p_1,p_2\in Per(g)$ such that $ind(p_1)\neq ind(p_2)$.
\end{lemma}

We say that a diffeomorphism $f$ is \textit{Kupka-Smale} if all periodic points of $f$ are hyperbolic and for all pair of periodic points $p$ and $q$ of $f$ the manifolds $W^s(p)$ and $W^u(q)$ are transversal. We denote by $\mathcal{KS}$ the set of all Kupka-Smale diffeomorphisms. It is well known that $\mathcal{KS}$ is a residual subset of $\mundo$.

We say that two hyperbolic periodic points $p$ and $q$ are \textit{homoclinicaly related} if $W^s(\mathcal{O}(p))\tcap W^u(\mathcal{O}(q))\neq\emptyset$ and $W^u(\mathcal{O}(p))\tcap W^s(\mathcal{O}(q))\neq\emptyset$. It is easy to see that two hyperbolic periodic points homoclinicaly related have the same index. As another corollary of Lemma \ref{limsha} we obtain:

\begin{corollary}\label{genlimsha}
If $f\in\mathcal{TLS}\cap\mathcal{KS}$ then all periodic points have the same index.
\end{corollary}

\begin{proof}
Let $p,q$ be two distinct periodic points of $f$. Lemma \ref{limsha} assures that the sets $W^s(p)\cap W^u(q)$ and $W^u(p)\cap W^s(q)$ are non-empty. As $f\in\mathcal{KS}$, these intersections are transversal. Thus $p$ and $q$ are homoclinicaly related and have the same index.
\end{proof}

\vspace{+0.4cm}

\hspace{-0.4cm}\textit{Proof of Theorem \ref{maintheorem2}}:
Suppose that $f\in int(\mathcal{TLS})$. Let $\mathcal{U}$ be a $C^1$ neighborhood of $f$ contained in $\mathcal{TLS}$. We claim that $f$ is a star diffeomorphism. Suppose by contradiction that $f$ is not star. Then there exist $g\in\mathcal{U}$ and $p$ a non-hyperbolic periodic point of $g$. We use Lemma \ref{pert} and obtain $h\in\mathcal{U}$ and two distinct hyperbolic periodic points $p,q$ of $h$ with different indices. As these points are hyperbolic and $\mathcal{KS}$ is dense in $\mundo$ we can perturb $h$ to $\tilde{h}\in\mathcal{U}\cap\mathcal{KS}$ that has two distinct hyperbolic periodic points $p_{\tilde{h}},q_{\tilde{h}}$ with different indices. This contradicts Corollary \ref{genlimsha} and proves the claim. Hayashi's Theorem \cite{H} implies that $f$ is hyperbolic and Theorem \ref{depois} finishes the proof.
\qed

\vspace{+0.3cm}

To prove Theorem \ref{maintheorem1} we first need to introduce some generic machinery. Inspired in the work of S. Gan and D. Yang (\cite{GY} Lemma 2.1) on $C^1$ generic expansive homoclinic classes we prove the following:

\begin{lemma}\label{difind}
There exists an open and dense set $\mathcal{P}$ of $\mundo$ such that all $f\in\mathcal{P}$ satisfies: if for any $C^1$ neighborhood $\mathcal{U}$ of $f$ some $g\in\mathcal{U}$ has two distinct hyperbolic periodic points with different indices then $f$ has two distinct hyperbolic periodic points with different indices.
\end{lemma}

\begin{proof}
Let $\mathcal{A}$ be the set of $C^1$ diffeomorphisms that has two distinct hyperbolic periodic points with different indices. The hyperbolicity of these periodic points imply that $\mathcal{A}$ is open in $\mundo$. Let $\mathcal{B}=\mundo\setminus\overline{\mathcal{A}}$. Note that $\mathcal{B}$ is also open in $\mundo$. Thus $\mathcal{P}=\mathcal{A}\cup\mathcal{B}$ is open and dense in $\mundo$.

Let $f\in\mathcal{P}$ and suppose that there is a sequence $(f_n)_{n\in\N}$ of diffeomorphisms converging to $f$ in the $C^1$ topology such that each $f_n$ has two distinct hyperbolic periodic points with different indices. Then $f_n\in\mathcal{A}$ for all $n\in\N$ implies that $f$ cannot belong to $\mathcal{B}$ and must belong to $\mathcal{A}$, i.e., $f$ has two distinct hyperbolic periodic points with different indices.
\end{proof}

\hspace{-0.4cm}\textit{Proof of Theorem \ref{maintheorem1}}:
Let $f\in\mathcal{TLS}\cap\mathcal{KS}\cap\mathcal{P}$. Suppose by contradiction that $f$ is not a star diffeomorphism, that is, for every neighborhood $\mathcal{U}$ of $f$ there exists some $g\in\mathcal{U}$ that has a non-hyperbolic periodic point. We perturb $g$ using Lemma \ref{pert} and obtain $h\in\mathcal{U}$ such that $h$ has two distinct hyperbolic periodic points with different indices. Since $f\in\mathcal{P}$ it has two distinct hyperbolic periodic points with different indices. But this contradicts Corollary \ref{genlimsha} and proves that $f$ is a star diffeomorphism. As above Theorem \ref{depois} finishes the proof. \qed

\section*{Acknowledgements}

The author would like to thank Alexander Arbieto, Bruno Santiago and Welington Cordeiro for the fruitful discussions during the preparation of the paper. Work partially supported by CAPES from Brazil.

\end{document}